\documentclass[12pt]{amsart}  % draft

\usepackage{custom_preamble} % showkeys

\begin{document}

\title{An analytic approach to a weak non-Abelian Kneser-type theorem}

\author{\tsname}
\address{\tsaddress}
\email{\tsemail}

\begin{abstract}
We prove the following result due to Hamidoune using an analytic approach.  Suppose that $A$ is a subset of a finite group $G$ with $|AA^{-1}| \leq (2-\varepsilon)|A|$.  Then there is a subgroup $H$ of $G$ and a set $X$ of size $O_\varepsilon(1)$ such that $A \subset XH$.
\end{abstract}

\maketitle

\section{Introduction}

In his blog (see also \cite{tao::8}) Tao asked for a non-Abelian version of Kneser's theorem and made a number of observations related to this as well as giving a conjectural form. This question was comprehensively answered by Hamidoune in \cite{ham::} using his isoperimetric method, but in these notes we shall describe a different, more analytic, approach.  As it happens this is not a particularly efficient idea and in that sense these notes may be more of a curiosity than an essential contribution.

We remark also that Tao has written a second later blog entry comparing the two approaches (Hamidoune's and the method here) and added a third qualitative explanation of the material below; it is certainly recommended if the reader is interested in this problem.

Suppose that $G$ is a (possibly non-Abelian) group and $A,B \subset G$.  We define the product set of $A$ and $B$ to be
\begin{equation*}
AB :=\{ab: a \in A, b \in B\}.
\end{equation*}
A coset of a subgroup in $G$ may be characterised as a non-empty set $H \subset G$ such that $|HH^{-1}| = |H|$.  Our interest lies in what happens when we relax the condition to consider non-empty sets $A$ such that $|AA^{-1}| \leq K|A|$, where $K> 1$.  When $G$ is Abelian sets of this form were studied by Fre{\u\i}man in his celebrated structure theory of set addition \cite{fre::,fre::3}, and recently there has been considerable interest in extending this work to the non-Abelian setting.

As indicated, if $H$ is a coset of a subgroup then $|HH^{-1}| =|H| \leq K|H|$ for any $K>1$.  On other other hand if $H$ is such and $A \subset H$ has $|A| \geq |H|/K$ then $|AA^{-1}| \leq |H| \leq K|A|$ since $AA^{-1} \subset HH^{-1}$ and $|HH^{-1}|=|H|$.  It turns out that for $K$ sufficiently small this is the \emph{only} way of constructing such sets $A$.  In particular we have the following result of Fre{\u\i}man \cite{fre::3}.
\begin{proposition}\label{prop.fre}
Suppose that $G$ is a group and $A \subset G$ has $|AA^{-1}| \leq K|A|$ for some $K<1.5$.  Then $A$ is contained in a (left) coset of a subgroup $H$ with  $|H| \leq K|A|$.
\end{proposition}
It is instructive to see the proof of this since the definitions and tools will be useful later.  In particular the proof motivates the introduction of convolution.

Suppose that $f,g \in \ell^1(G)$.  Then we define the convolution of $f$ and $g$ to be the function
\begin{equation*}
f \ast g(x):=\sum_{yz=x}{f(y)g(z)} \text{ for all } x \in G.
\end{equation*}
The convolution is useful for two important reasons: the first, if $A, B \subset G$ are finite then
\begin{equation*}
\supp 1_A \ast 1_B = AB
\end{equation*}
so that we can analyse the product set $AB$ through the convolution $1_A \ast 1_B$.  This is often rather easier to do than analysing $1_{AB}$ directly since the convolution is (typically) smoother.  

The second reason convolution is important is that
\begin{equation*}
1_A \ast 1_B(x) = |A \cap xB^{-1}| \text{ for all } x \in G.
\end{equation*}
To prove Proposition \ref{prop.fre} we need both of these facts.  If $x \in A^{-1}A$ then $x=a^{-1}a'$ for some $a,a' \in A$ and so using the second fact we see that
\begin{eqnarray}\label{eqn.key}
1_{A^{-1}} \ast 1_{A}(x) &= & |A^{-1} \cap (a^{-1}a'A^{-1})| = |aA^{-1} \cap a'A^{-1}|\\ \nonumber & \geq & |aA^{-1}| + |a'A^{-1}| - |aA^{-1} \cup a'A^{-1}| \geq (2-K)|A|;
\end{eqnarray}
the first fact tells us that that if $1_{A^{-1}} \ast 1_A(x) \neq 0$ then $x \in A^{-1}A$ and so $1_{A^{-1}} \ast 1_A(x) \geq (2-K)|A|$.  Crucially this leads to a step in the values $1_{A^{-1}} \ast 1_A$ may take which will also be useful later.
\begin{proof}[Proof of Proposition \ref{prop.fre}] Suppose that $x,y \in A^{-1}A$.  By (\ref{eqn.key}) there are more than $|A|/2$ pairs $(a,a')\in A \times A$ such that $x=a^{-1}a'$ and more than $|A|/2$ pairs $(a'',a''') \in A\times A$ such that $y=a''^{-1}a'''$.  It follows that there must be two pairs $(a,a')$ and $(a'',a''')$ with $a' = a''$, and hence $xy = a^{-1}a'a''^{-1}a'''=a^{-1}a''' \in A^{-1}A$.  It follows that $(A^{-1}A)^2 = A^{-1}A$, but $A^{-1}A$ is also symmetric (and non-empty) and so $A^{-1}A$ is a subgroup of $G$.  On the other hand for any $a \in A$ we have $a^{-1}A \subset A^{-1}A$ and so $A \subset aA^{-1}A$ and the result is proved.
\end{proof}
The restriction $K<1.5$ in this arugment is not simply an artefact of the proof: there is a qualitatively new structure which occurs at this threshold.  Suppose that $x$ is an element of order $4$ and consider $A:=\{1_G,x\}$.  Then $|AA^{-1}| = 3$ while $|A|=2$ so that $|AA^{-1}| \leq 1.5|A|$.  On the other hand if $H$ is a coset of a subgroup containing $A$ then $AA^{-1}\subset HH^{-1}$ and so $HH^{-1}$ contains $x$, an element of order $4$.  Thus the $HH^{-1}$, which is a group, has size at least $4$ and we conclude that the smallest coset containing $A$ has size at least $4$ which is bigger than $3=1.5|A|$.

While the set $A$ above cannot be very efficiently contained in a subgroup, it can be very efficiently covered by a subgroup: the trivial subgroup.  In light of this (and following Green and Ruzsa \cite{greruz::}, but see also Tao \cite{tao::9}) we say that a set $A$ is \emph{$K$-covered (on the left) by a set $B$} if there is a set $X$ of at most $K$ elements such that $A \subset XB$.  We shall then prove the following theorem.
\begin{theorem}[Weak non-abelian Kneser]\label{thm.ham}
Suppose that $G$ is a finite group and $A \subset G$ has $|AA^{-1}| \leq (2-\varepsilon)|A|$.  Then $A$ is $O_{\varepsilon}(1)$-covered by a subgroup $H$ of size at most $O_{\varepsilon}(|A|)$.
\end{theorem}
Note that this result is not a characterisation in the way that Proposition \ref{prop.fre} was.  Additionally the example of a long arithmetic progression shows that one cannot hope to remove the $\varepsilon$ entirely without expanding the class of structure one wishes to cover by.  This can be done, but is much harder than our work here; see \cite{bregretao::0} for details.

Before discussing our approach we note that this result is a corollary of the work of Hamidoune.  In \cite{ham::} he proved the following.
\begin{theorem}[{\cite[Theorem 1]{ham::}}]
Suppose that $G$ is a group and $A \subset G$ is finite.  Then there is a subgroup $H$ of $G$ such that at least one of the following holds:
\begin{enumerate}
\item $A^{-1}HA = A^{-1}A$ and $|A^{-1}A| \geq 2|HA| - |H|$;
\item $AHA^{-1}=AA^{-1}$ and $|AA^{-1}| \geq 2|AH| - |H|$.
\end{enumerate}
\end{theorem}
Hamidoune's theorem makes the connection with Kneser's theorem \cite{kne::} clearer; we record it now for comparison.  (A proof may also be found in \cite[Theorem 5.5]{taovu::}, which we mention for convenience as this is the standard text book for additive combinatorics.) 
\begin{theorem}[Kneser's theorem (symmetric version)]  Suppose that $G$ is an Abelian group and $A \subset G$ is finite.  Then there is a subgroup $H$ of $G$ such that
\begin{equation*}
A-A + H= A-A \text{ and }|A-A| \geq 2|A+H| - |H|.
\end{equation*}
\end{theorem}
Theorem \ref{thm.ham} follows from Hamidoune's result.
\begin{proof}[Proof of Theorem \ref{thm.ham}]
We apply Hamidoune's theorem to get a subgroup $H$ such that 
\begin{equation*}
AHA^{-1}=AA^{-1} \textrm{ and } |AA^{-1}| \geq 2|AH| - |H|.
\end{equation*}
Since $|AH| \geq |A|$ and $|AA^{-1}| \leq (2-\varepsilon)|A|$ we conclude from the inequality that $|H| \geq \varepsilon |A|$.  On the other hand if $A$ has non-empty intersection with $R$ left cosets of $H$, then $|AHA^{-1}| \geq R|H|$, whence $R \leq2\varepsilon^{-1}-1$.  It follows that $A$ is $2\varepsilon^{-1}$-covered by $H$ and the result is proved.
\end{proof}
Note that the dependence on $\varepsilon$ here is sharp up to the multiplicative constant as can be seen by considering an arithmetic progression of length about $\varepsilon^{-1}$.

Our approach to Theorem \ref{thm.ham} is based around the following idea much of which was also identified as important by Tao when he recorded the original question.

Roughly, we proceed by analysing $1_{A^{-1}} \ast 1_A$ which as a convolution is pretty smooth, but then we saw in (\ref{eqn.key}) that for $K<2$ there is a jump between when $1_{A^{-1}} \ast 1_A$ is zero and when it is non-zero.  These two facts mean that $A^{-1}A$, the support of $1_{A^{-1}} \ast 1_A$ must be a `connected component' in some sense which turns out to mean that it is a small union of cosets of a subgroup.  

\section{Analytic proof of Theorem \ref{thm.ham}}

Some readers may wish to proceed assuming that $G$ is Abelian to get a sense of how the argument goes, although obviously in this setting the usual version of Kneser's theorem is well-known and immediately yields Theorem \ref{thm.ham} by the same argument we used to derive it from Hamidoune's theorem in the general case.

We shall need two main results in our work.  The first is a non-Abelian Bogolyubov-Ruzsa-type result (\emph{c.f.} \cite{bog::,ruz::9}) from the paper \cite{crosis::} of Croot and Sisask.  This provides us with a set which is an approximate group in the sense of \cite{tao::6} and which is also correlated with our set $A$.
\begin{proposition}\label{prop.cs}  Suppose that $G$ is a group, $A \subset G$ is a finite set with $|AA^{-1}| \leq K|A|$ and $k \in \N$ is a parameter.  Then there is a symmetric neighbourhood of the identity, $X$, such that $|X|=\Omega_{K,k}(|A|)$ and
\begin{equation*}
1_{A^{-1}} \ast 1_{A} \ast 1_{A^{-1}} \ast 1_A(x)\geq |A|^3/2K \text{ for all }x \in X^k.
\end{equation*}
\end{proposition}
This is \cite[Theorem 4.1]{crosis::} applied to the sets $A^{-1}$, $A$ and $A$, and using the fact that 
\begin{equation}\label{eqn.sym}
\|1_{A^{-1}} \ast 1_A \|_{\ell^2(G)}^2\geq |A|^3/K
\end{equation}
if $|AA^{-1}| \leq K|A|$.  One should like to prove this by applying the Cauch-Schwarz inequality but this requires that $|A^{-1}A| \leq K|A|$ which is \emph{not} our hypothesis.  However, in \cite[Lemma 4.3]{tao::6} Tao saw that
\begin{equation*}
\|1_{A^{-1}} \ast 1_A \|_{\ell^2(G)}^2=\|1_{A} \ast 1_{A^{-1}} \|_{\ell^2(G)}^2
\end{equation*}
since $\langle f \ast g,h\rangle=\langle g, \tilde{f} \ast h\rangle = \langle f,h \ast \tilde{g}\rangle$ for all functions $f,g,h \in \ell^1(G)$. (\ref{eqn.sym}) then follows from Cauchy-Schwarz on the right hand quantity and the hypothesis $|AA^{-1}| \leq K|A|$.  

It may be worth noting that when $|AA^{-1}|<2|A|$ we have $AA^{-1}=A^{-1}A$, and so the above switch is not necessary.  Tao presented a proof of this fact in \cite{tao::8}, but it is more involved than the argument above so we have not recorded it here.

The paper \cite{crosis::} of Croot and Sisask is well worth reading and the proof of \cite[Theorem 4.1]{crosis::} (and hence Proposition \ref{prop.cs}) is not long, although it is rather clever.  One of the main points of their argument though is that they achieved good dependencies on $k$ and $K$, something we do not record as the second result we used is not blessed with such good dependencies.

We need a little notation: suppose that $G$ is a finite group and $X \subset G$.  Then we write $\P_X$ for the uniform probability measure supported on $X$.  Given a measure $\mu$ on $G$ we write $\tilde{\mu}$ for the measure assigning mass $\overline{\mu(\{x^{-1}\})}$ to each $x \in G$, and similarly for functions.  Finally, convolution of a function $f$ and a measure $\mu$ is defined point-wise by
\begin{equation*}
f \ast \mu(x) = \int{f(xz^{-1})d\mu(z)} \text{ for all } x \in G.
\end{equation*}
We can now state the result which is an easy corollary of \cite[Proposition 20.1]{san::9}.  (The result as stated in \cite{san::9} concerns functions in the Fourier-Eymard algebra but it is a short calculation (essentially in \cite[Lemma 6.1]{san::9}) to show that $f \ast \tilde{f}$ has algebra norm bounded by its $L^\infty$-norm.)
\begin{proposition}\label{prop.quantcon}
Suppose that $G$ is a finite group, $f \in \ell^2(G)$, $X$ is symmetric and $\P_G(X^4) \leq K\P_G(X)$ and $\nu \in (0,1]$ is a parameter. Then there are symmetric neighbourhoods of the identity $B' \subset B \subset X^4$ such that $\P_G(B')=\Omega_{K,\nu}(\P_G(X))$,
\begin{equation*}
\sup_{x \in G}{\|f \ast \tilde{f} \ast\widetilde{\P_B}\ast  \P_B-f \ast \tilde{f}\ast\widetilde{\P_B}\ast  \P_B(x)\|_{L^\infty(\P_{xB'})}}\leq \nu\|f \ast \tilde{f}\|_{L^\infty(G)}
\end{equation*}
and
\begin{equation*}
\sup_{x \in G}{\|f \ast \tilde{f} - f \ast \tilde{f}\ast\widetilde{\P_B}\ast  \P_B\|_{L^2(\P_{xB'})}}\leq \nu\|f \ast \tilde{f}\|_{L^\infty(G)}.
\end{equation*}
\end{proposition}
It is perhaps worth saying that very roughly this proposition makes quantitative the idea that if $f \in L^2(G)$ then $f \ast \tilde{f}$ is continuous.  This in itself is a little involved as the appropriate quantitative notion of continuity is (necessarily) not in $L^\infty$ but rather in a local $L^2$-norm.  The paper \cite{grekon::} was the first place to develop this idea in the Abelian context, and the above result is a non-Abelian extension localised to approximate groups.

\begin{proof}[Proof of Theorem \ref{thm.ham}]
We apply Proposition \ref{prop.cs} to the set $A$ with $k=8$ to get a symmetric neighbourhood of the identity, $X$, with $|X|=\Omega(|A|)$ such that
\begin{equation*}
1_{A^{-1}} \ast 1_{A} \ast 1_{A^{-1}} \ast 1_A(x)\geq |A|^3/4 \text{ for all }x \in X^k.
\end{equation*}
First this tells us that
\begin{equation*}
|X^4||A|^3/4 \leq \sum_{x \in G}{1_{A^{-1}} \ast 1_{A} \ast 1_{A^{-1}} \ast 1_A(x)} = |A|^4
\end{equation*}
which combined with the lower bound on $|X|$ gives $|X^4| = O(|X|)$.

Now apply Proposition \ref{prop.quantcon} to $f=1_{A^{-1}}$ with this set $X$ and parameter $\nu=\varepsilon/10$.  This tells us (on combining the two conclusions of the proposition using the triangle inequality) that
\begin{equation*}
\int{|1_{A^{-1}} \ast 1_{A}(y) - 1_{A^{-1}} \ast 1_{A}\ast\widetilde{\P_B}\ast  \P_B(x)|^2d\P_{xB'}(y)} \leq 4\nu^2|A|^2 \textrm{ for all } x \in G,
\end{equation*}
since $\|1_{A^{-1}} \ast 1_{A}\|_{L^\infty(G)} = |A|$.  Now, suppose for a contradiction that there is some $x \in G$ such that
\begin{equation}\label{eqn.intermediate}
\varepsilon |A|/4 < 1_{A^{-1}} \ast 1_{A} \ast \widetilde{\P_B}\ast  \P_B(x)< 3\varepsilon |A|/4,
\end{equation}
in which case
\begin{equation*}
\int{|1_{A^{-1}} \ast 1_{A}(y) - 1_{A^{-1}} \ast 1_{A}\ast\widetilde{\P_B}\ast  \P_B(x)|^2d\P_{xB'}(y)} >( \varepsilon/4)^2|A|^2
\end{equation*}
in light of (\ref{eqn.key}).  This contradicts our choice of $\nu$ and hence there are no $x \in G$ such that (\ref{eqn.intermediate}) holds.

On the other hand in light of the first conclusion in Proposition \ref{prop.quantcon}, for all $x \in G$ we have
\begin{equation*}
|1_{A^{-1}} \ast 1_{A}\ast\widetilde{\P_B}\ast  \P_B(xy)-1_{A^{-1}} \ast 1_{A}\ast\widetilde{\P_B}\ast  \P_B(x)|\leq \varepsilon|A|/10 \textrm{ for all } y \in B'.
\end{equation*}
Thus by the triangle inequality and the fact that (\ref{eqn.intermediate}) does not hold we conclude that
\begin{equation*}
S:=\{x \in G:  1_{A^{-1}} \ast 1_{A}\ast\widetilde{\P_B}\ast  \P_B(x)> 3\varepsilon|A|/4\}
\end{equation*}
is invariant under right multiplication by elements of $B'$, and hence by the group $H$ generated by $B'$.  Now suppose, for a contradiction, that $S$ is empty whence
\begin{eqnarray*}
\varepsilon|A|^3/4 & \geq & \langle 1_{A^{-1}} \ast 1_{A}\ast\widetilde{\P_B}\ast  \P_B, 1_{A^{-1}}\ast 1_A \rangle\\ & = & \langle 1_{A^{-1}} \ast 1_A \ast 1_{A^{-1}} \ast 1_{A},\widetilde{\P_B}\ast  \P_B \rangle.
\end{eqnarray*}
Of course, since $B \subset X^4$ we have that $\supp \widetilde{\P_B}\ast  \P_B \subset X^{-4}X^4=X^8$ and so
\begin{equation*}
\langle 1_{A^{-1}} \ast 1_A \ast 1_{A^{-1}} \ast 1_{A},\widetilde{\P_B}\ast  \P_B \rangle \geq |A|^3/4.
\end{equation*}
This leads to a contradiction if $\varepsilon$ is sufficiently small (which we may certainly assume) and so we conclude that $S$ is non-empty.  

Since $S$ is non-empty (and $H$ right invariant) we note that there is some $z \in G$ such that
\begin{eqnarray*}
3\varepsilon|A|/4& \leq & \langle \P_{zH},1_{A^{-1}} \ast 1_{A}\ast\widetilde{\P_B}\ast  \P_B\rangle\\ & = & \langle 1_A \ast \P_{zH}, 1_A \ast \widetilde{\P_B} \ast \P_B\rangle\\ & \leq & \|1_A \ast \P_{zH}\|_{\ell^\infty(G)}\|1_{A}\ast\widetilde{\P_B}\ast  \P_B\|_{\ell^1(G)} = \|1_A \ast \P_{zH}\|_{\ell^\infty(G)}|A|;
\end{eqnarray*}
we conclude that there is some $x \in G$ for which $|A \cap xH| \geq 3\varepsilon |H|/4$.  Given this we first note that $|H|=O_\varepsilon(|A|)$; secondly, since $|H| \geq |B'| = \Omega_{\varepsilon}(|A|)$, we have
\begin{equation*}
|x^{-1}A \cap H| =\Omega_\varepsilon(|A|).
\end{equation*}
Finally we decompose $G$ into left cosets of $H$, and suppose that there are $R$ cosets $yH$ with $|yH \cap A|>0$.  Then
\begin{equation*}
O(|A|)=|AA^{-1}|= |AA^{-1}x| \geq |A \cap (x^{-1}A \cap H)^{-1}| \geq R.\Omega_{\varepsilon}(|A|).
\end{equation*}
It follows that $R=O_{\varepsilon}(1)$ and hence $A$ is contained in $O_\varepsilon(|A|)$ left cosets of $H$ as required.
\end{proof}
The bounds in Proposition \ref{prop.quantcon} are very poor, but even in the Abelian setting they are at best exponential in $\nu^{-2}$.  This can be seen by examining the Niveau sets of Ruzsa \cite{ruz::6} (see also \cite{grekon::} and \cite{wol::0}).  This dependence means we necessarily get at best an exponential bound in $\varepsilon^{-2}$ in Theorem \ref{thm.ham}, whereas Hamidoune's work is far better giving a linear bound.

To summarise what we have seen: our method is a more complicated way of getting a weaker result which cannot ever yield a result as strong as Hamidoune's.

\section*{Acknowledgements}

The author should very much like to thank Terry Tao for conversations leading to these notes and for providing the preprint on Hamidoune's work; Emmanuel Breuillard and Ben Green for useful conversations; and the organisers of the conference `Combinatoire Additive {\'a} Paris 2012' in honour of Hamidoune which inspired the completion of these notes.

\bibliographystyle{halpha}

\bibliography{references}

\end{document}